\documentclass[11pt]{article}

\usepackage{amsmath}
\usepackage{amsthm}
\usepackage{amsfonts}
\usepackage{amssymb}
\usepackage{enumerate}
\usepackage{hyperref}
\usepackage[bottom,symbol]{footmisc}
\usepackage{tikz}
\usetikzlibrary{calc}
\title{On the codegree density of complete 3-graphs and related problems}
\author{Victor Falgas--Ravry
\thanks{Institutionen f\"or matematik och matematisk
statistik, Ume{\aa}  Universitet, 901 87 Ume{\aa}, Sweden. Supported by a
postdoctoral grant from the Kempe foundation. Email: {\tt
victor.falgas-ravry@math.umu.se}}
}

\theoremstyle{plain}
\newtheorem{theorem}{Theorem}

\newtheorem{corollary}[theorem]{Corollary}
\newtheorem{proposition}[theorem]{Proposition}
\newtheorem*{claim}{Claim}

\newtheorem{construction}[theorem]{Construction}

\theoremstyle{definition}
\newtheorem{definition}{Definition}
\newtheorem{conjecture}{Conjecture}
\newtheorem{question}[conjecture]{Question}
\newtheorem{problem}[conjecture]{Problem}

\theoremstyle{remark}
\newtheorem{remark}[theorem]{Remark}

\begin{document}
\maketitle

\begin{abstract}
Given a family of 3-graphs $\mathcal{F}$, its codegree threshold $\textrm{coex}(n, \mathcal{F})$ is the largest number $d=d(n)$ such that there exists an n-vertex 3-graph in which every pair of vertices is contained in at least $d$ 3-edges but which contains no member of $\mathcal{F}$ as a subgraph. The codegree density of $\mathcal{F}$ is the limit 
\[\gamma(\mathcal{F})=\lim_{n\rightarrow \infty} \frac{\textrm{coex}(n,\mathcal{F})}{n-2}.\]

In this paper we generalise a construction of Czygrinow and Nagle to bound below the codegree density of complete 3-graphs: for all integers $s\geq 4$, the codegree density of the complete 3-graph on $s$ vertices $K_s$ satisfies
\[\gamma(K_s)\geq 1-\frac{1}{s-2}.\]
We also provide constructions based on Steiner triple systems which show that if this lower bound is sharp, then we do not have stability in general.

In addition we prove bounds on the codegree density for two other infinite families of 3-graphs.
\end{abstract}
\section{Introduction}
In this paper, we study codegree density for various families of 3-graphs.

\subsection{Notation and definitions}
Given a set $A$ and an integer $r$, write $A^{(r)}$ for the collection of $r$-subsets of $A$. Also, for $n \in \mathbb{N}$ write $[n]$ for $\{1,2,\ldots m\}$.

A \emph{3-graph} is a pair $G=(V,E)$, where $V=V(G)$ is a set of vertices and $E=E(G)\subseteq V^{(3)}$ is a set of 3-edges. A \emph{subgraph} of $G$ is a 3-graph $H$ with $V(H) \subseteq V(G)$ and $E(H)\subseteq E(G)$. The \emph{codegree} $d(x,y)$ of vertices $x,y \in V(G)$ is the number of 3-edges of $G$ containing the pair $\{x,y\}$. The \emph{minimum codegree} of $G$ is $\delta_2(G)=\min_{(x,y) \in V^{(2)}}d(x,y)$.

We shall also consider some \emph{2-graphs}, or ordinary graphs, which are pairs $G=(V,E)$, with $E$ now a set of (2-)edges, $E\subseteq V^{(2)}$.

We recall the classical definition of the Tur\'an number and Tur\'an density of a family of 3-graphs.
\begin{definition}
Let $n \in \mathbb{N}$, and let $\mathcal{F}$ be a family of non-empty 3-graphs. The \emph{Tur\'an number} $\textrm{ex}(n, \mathcal{F})$ of $\mathcal{F}$  is the largest number $e=e(n)$ such that there exists an $n$-vertex $3$-graph with at least $e$ 3-edges and no member of $\mathcal{F}$ as a subgraph.
\end{definition}
An easy averaging argument shows that the sequence $\textrm{ex}(n, \mathcal{F})/\binom{n}{3}$ is monotone decreasing in $[0,1]$, and hence converges to a limit, known as the Tur\'an density.
\begin{definition}
The \emph{Tur\'an density} of a family of non-empty 3-graphs $\mathcal{F}$ is the limit
\[\pi(\mathcal{F})=\lim_{n \rightarrow \infty} \frac{\textrm{ex}(n, \mathcal{F})}{\binom{n}{3}}.\]
\end{definition}
The Tur\'an density may be thought of as the asymptotically maximal proportion of 3-edges which may be present in an $\mathcal{F}$-free 3-graph, and is one of the central objects of study in extremal hypergraph theory. In this paper, we are interested in another limit density, namely codegree density.

\begin{definition}
Let $n \in \mathbb{N}$ and let $\mathcal{F}$ be a family of non-empty 3-graphs. The \emph{codegree threshold} of $\mathcal{F}$ $\textrm{coex}(n, \mathcal{F})$  is the largest number $d=d(n)$ such that there exists an $n$-vertex $3$-graph in which every pair of vertices is contained in at least $d$  3-edges but which contains no member of $\mathcal{F}$ as a subgraph.
\end{definition}

Mubayi and Zhao~\cite{MubayiZhao07} showed that for any $\mathcal{F}$, the sequence $\textrm{coex}(n,\mathcal{F})/(n-2)$ tends to a limit as $n \rightarrow \infty$. (Note that this sequence is not monotone, as shown by Lo and Markstr\"om~\cite{LoMarkstrom12}, so that the existence of a limit is not trivial.) This allows us to define the codegree density of a family of $3$-graphs.

\begin{definition}
The \emph{codegree density} of a family of non-empty $3$-graphs $\mathcal{F}$ is defined to be the limit
\[\gamma(\mathcal{F})=\lim_{n \rightarrow \infty}  \frac{\textrm{coex}(n,\mathcal{F})}{n-2}.\]
\end{definition}

\subsection{History}
Codegree density was first studied by Mubayi~\cite{Mubayi05}, who determined it for the Fano plane $F_7$. Keevash~\cite{Keevash09} used hypergraph regularity to show $\textrm{coex}(n, F_7)=\lfloor n/2 \rfloor$ for $n$ sufficiently large, with the unique extremal configuration a complete balanced bipartite 3-graph. De Biasio and Jiang \cite{DeBiasioJiang12} later gave another proof of this fact avoiding the use of hypergraph regularity.

Mubayi and Zhao~\cite{MubayiZhao07} showed that codegree densities are well-defined and studied various properties of $\gamma$. In particular they showed that for every $c \in [0,1]$ there exists a family of 3-graphs $\mathcal{F}$ with $\gamma(\mathcal{F})=c$ (so that codegree density does not `jump'), and that the `supersaturation' phenomenon familiar from extremal graph theory also occurs for codegree density. (See~\cite{MubayiZhao07} for details and definitions.)

Marchant, Pikhurko, Vaughan and the author~\cite{falgasravry+marchant+pikhurko+vaughan:arxiv} determined the codegree threshold of $F_{3,2}=([5], \{123,124,125,345\})$, while Pikhurko,Vaughan and the author determined the codegree density of $K_4^-=([4], \{123,124,134\})$, resolving a conjecture of Nagle~\cite{Nagle99}.

Nagle~\cite{Nagle99} and Czygrinow and Nagle~\cite{CzygrinowNagle01} have in addition conjectured that $\gamma(K_4)=1/2$, where $K_4$ denotes the complete 3-graph on 4 vertices. We describe their lower-bound construction below.
\begin{construction}[Czygrinow and Nagle's construction]
Let $n \in \mathbb{N}$. Let $T$ be a tournament (an orientation of the edges of the complete 2-graph) on $[n]$. 
We define a 3-graph $G_T$ on $[n]$ by setting $ijk$ with $i<j<k$ to be a 3-edge of $G_T$ if the ordered pairs $(i,j)$ and $(i,k)$ receive opposite orientations in $T$.
\end{construction} 
It is easy to check that $G_T$ has no $K_4$ subgraph and that by choosing $T$ uniformly at random we obtain a 3-graph that with high probability has minimum codegree at least $n/2-o(n)$. Czygrinow and Nagle conjectured that this was asymptotically best possible, in other words that $\gamma(K_4)=1/2$.

No other codegree densities are known or conjectured, and like Tur\'an's famous conjecture that $\pi(K_4)=5/9$, the Czygrinow--Nagle conjecture remains wide open. We refer a reader to Keevash's recent survey~\cite{Keevash11} for a more complete discussion of Tur\'an-type problems for 3-graphs.

\subsection{Contribution of this paper}
In this note we give a general construction showing that 
\[\gamma(K_s)\geq 1- \frac{1}{s-2}\]
for all $s\geq 4$, where $K_s$ denotes the complete 3-graph on $s$ vertices, $K_s=([s], [s]^{(3)})$. Our construction is a generalization of the Czygrinow--Nagle construction based on random edge-colourings of the complete 2-graph.

In addition, for $s$ congruent to $1$ or $5$ modulo $6$, we give different non-isomorphic constructions giving the same lower bound on $\gamma(K_s)$. These are based on Steiner triple systems, and imply that if our lower-bound is tight (as we believe) then the codegree density problem for complete 3-graphs is not stable in general: several very different near-extremal configurations exist. This mirrors the conjectured behaviour of Tur\'an density for complete 3-graphs (see~\cite{Keevash11,Sidorenko95}). In the particular case $s=6$, we are also able to give an alternative random construction based on Ramsey numbers showing $\gamma(K_6)\geq 3/4 $.

Finally we also give bounds on the codegree and Tur\'an densities of two other families of 3-graphs and present a number of open problems.

Our paper is structured as follows. In Section 2, we prove lower-bound on the codegree density of complete 3-graphs. In Section 3, we turn our attention to 3-graphs of the form $\left([t]\sqcup \{x_{\star}\}, \{ijx_{\star}: \ 1\leq i<j \leq t \}\right)$ (these correspond to complete 2-graphs on $t$ vertices in the links of the vertices --- see Section 3 for a formal definition) and give general bounds for both their codegree density and their Tur\'an density. Finally in Section 4 we introduce co-spanned 3-graphs, and give bounds on their codegree density.

\section{Complete 3-graphs}
Let $n \in \mathbb{N}$ and $[n]=\{1,2 , \ldots n\}$.
\begin{construction}[Colouring construction]\label{colourconstr}
Let $c: [n]^{(2)}\rightarrow [s]$ be a colouring of the edges of the complete 2-graph on $[n]$ with $s$ colours. We construct a 3-graph $G_c$ based on this colouring in the following manner: for every triple $i,j,k\in [n]$ with $i<j<k$, we add the 3-edge $ijk$ to $E(G_c)$ if and only if $c(ij)\neq c(ik)$.
\end{construction}
\begin{remark}
This may naturally be viewed as a generalisation of the Czygrinow--Nagle construction: we may obtain a 2-colouring of $[n]^{(2)}$ from a tournament $T$ on $[n]$ by setting $c(ij)=1$ for $i<j$ if $ij$ is oriented from $i$ to $j$ in $T$, and $c(ij)=2$ if instead $ij$ is oriented from $j$ to $i$.
\end{remark}

\begin{proposition}\label{constructionworks}
For any colouring $c$ as above, the 3-graph $G_c$ is $K_{s+2}$-free.
\end{proposition}
\begin{proof}
Let $i_1, i_2, \ldots i_{s+2}$ be a set of $s+2$ distinct vertices from $[n]$ with $i_1<i_2\ldots <i_{s+2}$.

\begin{itemize}
\item Suppose all 3-edges of the form $i_1i_2i_j$ with $2<j\leq s+2$ are in $G_c$. Then we must have that $c(i_1i_j) \in [s]\setminus\{c(i_1i_2)\}$ for all $j$ with $2<j\leq s+2$.

\item Suppose in addition all 3-edges of the form $i_1i_3i_j$ with $3<j \leq s+2$ are in $G_c$. Then we must have that $c_{i_1i_j} \in [s] \setminus \{c(i_1i_2), c(i_1i_3)\}$ for all $j$ with $3<j\leq s+2$.

\item Repeating the argument by supposing all 3-edges of the form $i_1i_4i_j$ ($4<j\leq s+2)$, $i_1i_5i_j$ ($5<j\leq s+2$), ... , $i_1i_{s}i_j$ ($s<j \leq s+2$) are in $G_c$, we have that both of $c(i_1i_{s+1})$ and $c(i_1 i_{s+2})$ lie in the set
\[S=[s] \setminus \{c(i_1i_j)\vert 2\leq j \leq s\}.\]
\end{itemize}
Since $c(i_1i_j)\neq c(i_1i_{j'})$ for all $j,j'$ with $2\leq j<j'\leq s$, we have that $S$ has size one. Thus $c(i_1i_{s+1})=c(i_1i_{s+2})$, and the 3-edge $i_1i_{s+1}i_{s+2}$ is missing from $G_c$.

It follows that whatever colouring $c$ we originally chose, at least one of the $3$-edges $i_1i_ji_{j'}$ with $2\leq j <j'\leq s+2$ is missing from $G_c$. In other words, $i_1, i_2 \ldots i_{s+2}$ cannot span a complete $3$-graph in $G_c$, which is therefore $K_{s+2}$-free as claimed.
\end{proof}

\begin{theorem}\label{codegreebound}
For all integers $s\geq 2$,
\[\gamma(K_{s+2})\geq 1- \frac{1}{s}.\]
\end{theorem}
\begin{proof}
Let $n$ be sufficiently large. Independently colour each pair from $[n]$ with an element of $[s]$ chosen uniformly at random, and let $\mathbf{c}$ denote the random colouring thus obtained.

Consider now the $3$-graph $G_{\mathbf{c}}$. By Proposition~\ref{constructionworks}, we know it is $K_{s+2}$-free. We show that with high probability it has minimum codegree 
\[\delta_2(G_{\mathbf{c}})=\left(1-\frac{1}{s}\right) n +o(n).\]

For each pair $ij \in [n]^{(2)}$ with $i<j$, let $X_{ij}$ be the random variable denoting the codegree of $i,j$ in $G_{\mathbf{c}}$. Further for every $k \in[n]\setminus\{ij\}$, let $X_{ij,k}$ be the Bernoulli random variable taking the value $1$ if $ijk \in E(G_{\mathbf{c}})$ and $0$ otherwise.

\begin{claim}
Fix $i<j$. Then $\{X_{ij,k}:\ k \in [n]\setminus\{ij\} \}$ forms a family of independent identically distributed Bernoulli random variables with parameter $1-\frac{1}{s}$.
\end{claim}
\begin{proof}
Let $K\sqcup K'$ be a partitition of $[n]\setminus\{ij\}$. Then, 
\begin{align*}
\mathbb{P}&\left(X_{ij,k}= 1\  \forall k \in K, \ X_{ij, k'}=0\  \forall k' \in K'\right)\\
&=\sum_{c=1}^s \mathbb{P}(\mathbf{c}(ij)=c) \mathbb{P}\left(X_{ij,k}= 1\  \forall k \in K, \ X_{ij, k'}=0\  \forall k' \in K'\right \vert \mathbf{c}(ij)=c)\\
&=\sum_{c=1}^s \frac{1}{s} \prod_{k\in K} \mathbb{P}(X_{ij,k}= 1\vert \mathbf{c}(ij)=c)\prod_{k'\in K'} \mathbb{P}(X_{ij,k'}= 0\vert \mathbf{c}(ij)=c)\\
&=\prod_{k\in K} \mathbb{P}(X_{ij,k}= 1\vert \mathbf{c}(ij)=1)\prod_{k'\in K'} \mathbb{P}(X_{ij,k'}= 0\vert \mathbf{c}(ij)=1).
\end{align*}
Here in the second equality above we used the fact that conditional on the value of $\mathbf{c}(ij)$, the random variable $X_{ij,k}$ depends only on the random variables $\mathbf{c}(ik), \mathbf{c}(jk)$. Since each edge is coloured independently, we have that the conditional random variables $X_{ij,k}\vert \mathbf{c}(ij)=c$ are independent. In the third equality, we use the fact that the problem is symmetric with respect to our $s$ colours.

Now for any $k\in K$, 
\begin{align*}
\mathbb{P}(X_{ij,k}=1)&=\sum_{c=1}^s \mathbb{P}(\mathbf{c}(ij)=c) \mathbb{P}(X_{ij,k}=1\vert \mathbf{c}(ij)=c)\\
&= s \times \frac{1}{s} \mathbb{P}(X_{ij,k}=1\vert \mathbf{c}(ij)=1),
\end{align*}
again using the symmetry in the colours, and similarly for any $k'\in K'$
\[\mathbb{P}(X_{ij,k'}=0)=\mathbb{P}(X_{ij,k'}=0\vert \mathbf{c}(ij)=1).\]
Thus we have 
\[\mathbb{P}\left(X_{ij,k}= 1\  \forall k \in K, \ X_{ij, k'}=0\  \forall k' \in K'\right)=\prod_{k\in K} \mathbb{P}(X_{ij,k}= 1)\prod_{k'\in K'} \mathbb{P}(X_{ij,k'}= 0),\]
for any $K\sqcup K'=[n]\setminus\{ij\}$, proving that $\{X_{ij, k}: k \in [n]\setminus\{ij\}\}$ forms a family of independent random variables as claimed. By construction, they are identically distributed Bernoulli random variables with parameter $(1-\frac{1}{s})$, completing our claim.
\end{proof}

We can now apply a standard Chernoff bound. Fix $\varepsilon>0$.
\begin{align*}
\mathbb{P}\left(\delta_2(G_{\mathbf{c}})\leq (1-\frac{1}{s}-\varepsilon)n\right)&\leq \binom{n}{2} \mathbb{P}\left(\textrm{d}_{G_{\mathbf{c}}}(1,2)\leq(1-\frac{1}{s}-\varepsilon)n\right)\\
&\leq n^2 e^{-\frac{\varepsilon^2}{2} n}\\
&=o(1).
\end{align*}
Thus for a typical colouring $\mathbf{c}$, the minimum codegree of $G_{\mathbf{c}}$ is at least $(1-\frac{1}{s}-\varepsilon)n$. Since $\varepsilon>0$ was arbitrary, it follows that 
\[\gamma(K_{s+2})\geq 1-\frac{1}{s},\]
as claimed.

\end{proof}

Other constructions are possible. For example, for $s$ congruent to $3$ or $5$ modulo $3$, we have rather different, structured constructions in addition to the random constructions arising from Construction~\ref{colourconstr}. Thus if our lower-bound on the codegree density of $K_{s+2}$ is tight then in general we do not have stability for the codegree densities of complete 3-graphs.

\begin{construction}[A Steiner triple system construction]
Let $s\geq 5$ be an integer congruent to $3$ or $5$ modulo $6$. Let $S$ be a Steiner triple system on $[s-2]$ --- that is, a 3-graph on $[s-2]$ such that every pair of vertices is contained in exactly one 3-edge. Such systems are known to always exist, subject to the aforementioned modulo 6 condition~\cite{Kirkman1847}.

Given $n\in \mathbb{N}$, let $\sqcup_{i=1}^{s-2}V_i$ be a balanced $(s-2)$-partition of $[n]$. We define a 3-graph $G_S$ on the vertex set $[n]$ by taking the following triples to form the 3-edge set:
\begin{itemize}
\item all triples of the form $V_iV_iV_j$ for distinct $i,j \in [s-2]$
\item all triples of the form $V_{i}V_{j}V_{k}$ for distinct $i,j,k \in [s-2]$ such that $ijk$ does not belong to $S$.
\end{itemize}
\end{construction}

\begin{proposition}
For any $n\in\mathbb{N}$, $s$ congruent to $3$ or $5$ modulo $6$ and any Steiner triple system $S$ on $[s-2]$, the 3-graph $G_S$ is $K_{s}$-free and has minimum codegree $(1-\frac{1}{s-2})n +O(1)$.
\end{proposition}
\begin{proof}
First of all let us establish $G_S$ is $K_s$-free. Note that as we have no 3-edges of the form $V_iV_iV_i$, any $s$-set of vertices meeting some part $V_i$ in at least 3 vertices cannot span a $K_s$. Now any $s$-set meeting no part in more than 2 vertices must meet at least $\lceil s/2\rceil$ parts.

We claim that every set of at least $\lceil s/2\rceil$ vertices from $[s-2]$ must span at least one 3-edge of $S$. Indeed label such a set as $X=\{x_1, x_2, \ldots, x_{\lceil s/2\rceil}\}$. Suppose for contradiction that $X$ is an independent set in $S$. Then since $S$ is a Steiner triple system, for each of $x_i$, $i=2, \ldots \lceil s/2\rceil$, there exists a unique $y_i$ such that $x_1x_iy_i$ is a 3-edge of $S$, and moreover these $y_i$ are distinct (else the pair $x_1y_i$ would be contained in more than one 3-edge). Thus we would need 
\begin{align*}
\left\lceil \frac{s}{2}\right\rceil-1 &\leq \left\vert [s-2]\setminus X \right\vert=s- \left\lceil \frac{s}{2}\right\rceil -2,
\end{align*}
a contradiction.

Thus any $s$-set of vertices meeting at least $\lceil s/2\rceil$ different parts $V_i$ must meet three distinct part $V_{i},V_{j},V_{k}$ such that $iijk$ is a 3-edge of $S$. By construction, no triple of the form $V_{i}V_{j}V_{k}$ is a 3-edge of $G_S$, and thus our $s$-set misses at least one 3-edge. It follows that $G_S$ is $K_s$-free, as claimed.

Now let us compute its codegree. Consider two vertices $v,v'$ of $G_S$. If they belong to the same part $V_i$, then for every vertex $w \notin V_i$, $vv'w$ is a 3-edge of $G_S$, and thus the codegree of $v$ and $v'$ is at least $(1-\frac{1}{s-2})n+O(1)$ (since our partition was balanced).

On the other hand suppose that $v \in V_i$ and $v'\in V_{i'}$ for some distinct $i,i'\in [s-2]$. Then for any $w\in V_i \sqcup V_{i'}$, $vv'w$ is a 3-edge of $G_S$. In addition, let us denote by $i''$ the unique member of $[s-2]$ such that $ii'i''$ is a 3-edge of the Steiner triple system $S$. Then for all $j \in [s-2]\setminus\{i,i',i''\}$ and all $w \in V_j$, $vv'w$ is a 3-edge of $G_S$. Thus the codegree of $v,v'$ is again at least $(1-\frac{1}{s-2})n+O(1)$.
\end{proof}

\begin{remark}
The construction above based on Steiner triple system gives a rather large number of non-isomorphic constructions: while the Fano plane is the unique (up to isomorphism) Steiner triple system on 7 vertices, there are for example 11,084,874,829 non-isomorphic Steiner triple systems on 19 vertices (see Section~4.5 in~\cite{ColbournDinitz10}).
\end{remark}

For $s=6$, we can give another construction based on random colourings and the Ramsey number of the triangle.
\begin{construction}[A Ramsey-based construction]\label{ramseyconstr}
Let $n\in \mathbb{N}$. Given a colouring $c:\ [n]^{(2)}\rightarrow\{1,2\}$, let $G_c$ be the 3-graph on $[n]$ with 3-edges consisting of all triples not inducing a monochromatic triangle with respect to $c$.
\end{construction}
Since the Ramsey number for monochromatic triangles in 2-coloured graphs is $R(3,3)=6$, every 6-set of vertices in $G_c$ must be missing at least one 3-edge, so that $G_c$ is $K_6$-free as required. The expected codegree of a pair of vertices in a typical colouring $c$ is $\frac{3}{4}(n-1)$. Applying the same probabilistic tools as in Theorem~\ref{codegreebound}, we can easily recover from this another proof of $\gamma(K_6)\geq 3/4$.

\begin{remark}
Typical instances of Construction~\ref{ramseyconstr} and Construction~\ref{colourconstr} with $s=4$ are genuinely different. Indeed let $c$ and $c'$ be respectively a 2- and a 4-colouring of $[n]^{(2)}$, chosen uniformly at random. Consider a 5-set $U=\{u_1,u_2, u_3, u_4, u_5\}$ from $[n]$ with $1\leq u_1 < u_2 <\ldots <u_5\leq n$.

The probability that $U$ induces a copy of $K_5$ in $G_c$ is the probability that $c$ decomposes the pairs from $U$ into two monochromatic 5-cycles, one with colour 1 and the other with colour 2. This occurs with probability 
\[\#\{\textrm{decompositions}\}\times 2^{-10}=3 \times 2^{-7}.\]
On the other hand for $i=1,2,3$, let $A_i$ be the event that for every $j,j': \ i<j<j'\leq 5$, $c'$ assigns different colours to $u_iu_j$ and $u_iu_{j'}$; this is exactly the event that all 3-edges of the form $u_iu_ju_{j'}$ with $i<j<j'$ are in the 3-graph $G_{c'}$ obtained from $c'$ by applying Construction~\ref{colourconstr}.

Observe that the characteristic functions of the $A_i$ events form a family of independent random variables, since they depend on the colours assigned by $c'$ to disjoint edge-sets, and that the probability that $U$ induces a copy of $K_5$ in $G_{c'}$ is exactly the probability that $\bigcap_{i=1}^3 A_i$ occurs. This probability is thus
\begin{align*}
\mathbb{P}(A_1)\mathbb{P}(A_2)\mathbb{P}(A_3)&= \frac{4!}{4^4} \frac{4!/1!}{4^3}\frac{4!/2!}{4^2}\\
&= \frac{3^3}{2^{10}}.
\end{align*}
It follows in particular that $K_5$ subgraphs have different frequencies in $G_c$ and $G_{c'}$ for typical $c$, $c'$. The lower-bound constructions for $\gamma(K_6)$ arising from Construction~\ref{colourconstr} and 
Construction~\ref{ramseyconstr} are thus genuinely different.
\end{remark}

An obvious question to ask is whether the lower bound we provide is tight.
\begin{question}\label{areboundstight}
Is $\gamma(K_s)=1-\frac{1}{s-2}$ for all $s\geq 4$?
\end{question}
For comparison, let us note that the Tur\'an density of $K_s$ is conjectured to be $1-\frac{4}{(s-1)^2}$, with many known non-isomorphic constructions attaining that bound~\cite{Keevash11}. However the standard examples contain pairs of vertices with codegree only $1- \frac{2}{s-1}$ proportion of the maximum possible.

It would also be interesting to know whether the Czygrinow--Nagle construction is the only example of a $K_4$-free 3-graph with codegree density $1/2+o(1)$.
\begin{question}
Are all $K_4$-free configurations with codegree density $1/2+o(1)$ `close' to a Czygrinow--Nagle construction?
\end{question}

\section{Suspensions of complete 2-graphs}

Let $G$ be a 3-graph. Given a vertex $x\in V(G)$, we may form a 2-graph in a natural way by considering the pairs of vertices $v,v'$ making a 3-edge with $x$ in $G$.
\begin{definition}
The \emph{link graph} of $x\in V(G)$ is the 2-graph
\[G_x=\left(V(G)\setminus \{x\}, \{vv': \ xvv' \in E(G)\} \right).\]
\end{definition}
For every 2-graph $H$ we can consider the 3-graph corresponding to the presence of an $H$-subgraph in a link-graph.
\begin{definition}
Given a 2-graph $H$, let $S(H)$ denote the \emph{suspension} of $H$, that is, the 3-graph with vertex set $V(H)\sqcup\{x_{\star}\}$ and 3-edges 
\[\{x_{\star}vv': \ vv' \in E(H)\}.\]
\end{definition}
So for example the complete 3-graph on 4 vertices with one 3-edge removed, $K_4^-$, may be thought of as the suspension $S(K_3^{(2)})$ of the ordinary triangle $K_3^{(2)}$.

As mentioned in the introduction, Pikhurko, Vaughan and the author~\cite{FalgasRavryPikhurkoVaughan13} recently showed $\gamma(K_4^-)=1/4$, with the extremal configuration obtained by taking as the 3-edges the oriented triangles in a random tournament. It is rather natural to ask what $\gamma(S(K_s^{(2)}))$ may be in general, where $K_s^{(2)}$ is the complete graph on $s$ vertices, 
\[K_s^{(2)}=([s], \{ij:  1\leq i <j \leq s\}).\]
\begin{problem}\label{suspensionprob}
Give bounds for $\gamma(S(K_s^{(2)}))$.
\end{problem}
Note that by Tur\'an's theorem~\cite{Turan41} and averaging we have $\gamma(S(K_s^{(2)}))\leq 1- \frac{1}{s-1}$. We give below a construction (which we do not believe is sharp in general) which shows that this trivial upper bound is not off by more than a multiplicative factor of $1-\frac{1}{s-2}$.
\begin{construction}[Rainbow triangles]\label{rainbowconstr}
Let $c: [n]^{(2)}\rightarrow [s-1]$ be a colouring of the edges of the complete 2-graph on $[n]$ with $s-1$ colours. We construct a 3-graph $G_c$ based on this colouring in the following manner: for every triple $i,j,k\in [n]$, add the 3-edge $ijk$ to $E(G_c)$ if and only if $c(ij), c(ik), c(jk)$ are pairwise distinct --- that is, if each of the sides of the triangle $ijk$ receives a different colour. We call such triangles `rainbow'.
\end{construction}
\begin{proposition}
For every colouring $c$, the 3-graph $G_c$ is $S(K_s^{(2)})$-free.
\end{proposition}
\begin{proof}
Let $v_0, v_1, \ldots, v_s$ be an $s+1$-set of vertices from $[n]$. Then by the pigeon-hole principle, there exist $1<i<j\leq s$ such that $c(v_0v_i)=c(v_0v_j)$. It follows that the triangle $v_0v_iv_j$ is not rainbow, and hence that $v_1, v_2\ldots v_s$ do not induce a copy of $K_s^{(2)}$ in the link graph of $v_0$ in $G_c$. Since $v_0, v_1, \ldots, v_s$ was arbitrary, it follows that $G_c$ is $SK_s^{(2)}$-free, as claimed.
\end{proof}
\begin{corollary}
\[\gamma(S(K_s^{(2)}))\geq (1-\frac{1}{s-1})(1-\frac{2}{s-1}).\]
\end{corollary}
\begin{proof}
Picking $c$ uniformly at random and applying Construction~\ref{rainbowconstr}, we have that the expected codegree of any pair of vertices in $G_c$ is $(1-\frac{1}{s-1})(1-\frac{2}{s-1})(n-2)$. Applying the same probabilistic tools as in Theorem~\ref{codegreebound}, we obtain from this a proof that $\gamma(S(K_s^{(2)}))\geq (1-\frac{1}{s-1})(1-\frac{2}{s-1})$.
\end{proof}
\begin{remark}
Since $\gamma(S(K_3^{(2)}))=\gamma(K_4^-)=\frac{1}{4}$, we know this bound on $\gamma(S(K_s^{(2)}))$ fails to be sharp for $s=3,4$. Given this, it seems unlikely that this construction is sharp for $s\geq 5$.
\end{remark}

We note that the analogue of Problem~\ref{suspensionprob} for Tur\'an density is also open.
\begin{problem}
Give bounds for $\pi(S(K_s^{(2)}))$.
\end{problem}
The Tur\'an density of $S(K_3^{(2)})=K_4^-$ and $S(K_4^{(2)})$ are conjectured to be $2/7$ and $1/2$ respectively, with the lower-bounds coming from recursive constructions due to Frankl and F\"uredi~\cite{FranklFuredi84} and Bollob\'as, Leader and Malvenuto~\cite{BollobasLeaderMalvenuto11} respectively. Close to matching upper bounds were obtained using flag algebras by Vaughan and the author~\cite{FalgasRavryVaughan13}, suggesting the lower bounds are best possible.

Below we give a generalisation of Bollob\'as, Leader and Malvenuto's construction for all integers $s\geq 4$ which are not divisible by $3$, which we conjecture is best possible.
\begin{construction}[Iterated complements of Steiner triple systems]\label{iteratedsteiner}
Let $s\geq 2$ be an integer congruent to $1$ or $2$ modulo $3$. Then $2s-1$ is congruent to $1$ or $3$ modulo $6$. Let $S$ be a Steiner triple system on $[2s-1]$ --- such a system is known to always exist, subject to the aforementioned modulo $6$ condition~\cite{Kirkman1847}.

Given $n \in \mathbb{N}$, we construct a 3-graph $G_S$ in an iterated fashion as follows. First of all, take a balance partition of $[n]$ into $2s-1$ parts $V_1, V_2 \ldots V_{2s-1}$. Now take as 3-edges all triples of the form $V_{i}V_{j}V_{k}$ with $1\leq i<j<k\leq 2s-1$ and $ijk \notin E(S)$. (This is equivalent to taking a blow-up of the complement of $S$.) Finally, repeat this construction inside each of the $2s-1$ parts $V_1, V_2 \ldots V_{2s-1}$.
\end{construction}
\begin{proposition}
The 3-graph $G_s$ is $S(K_s^{(2)})$-free and contains $\left(1-\frac{2}{s}\right)\binom{n}{3}+O(n^2)$ 3-edges.
\end{proposition}
\begin{proof}
Suppose for contradiction that we have an $s+1$-set of vertices $v_0, v_1 \ldots v_{s}$ in $G_S$ such that $v_1, \ldots v_{s}$ span a copy of $K_s^{(2)}$ in the link-graph of $v_0$.

Without loss of generality we may assume that $v_0\in V_{2s-1}$. We may also assume that at least one vertex $v_i$, $1\leq i \leq s$, lies in a different part from $v_0$ --- without loss of generality let us say we have $v_1 \in V_1$.

Since $S$ is a Steiner triple system, for every $i\in [2s-2]$, we have a unique $j \in [2s-2]\setminus\{j\}$ such that $ij(2s-1)$ is a 3-edge of $S$. This defines a matching on $[2s-2]$; by relabelling the parts if necessary, we may assume that this matching is $i(s-1+i): \ i \in [s-1]$.

Thus if $u\in V_i$ and $u'\in V_{s-1+i}$ then $uu'v_0$ it not a 3-edge of $G_S$. By construction, if $u \in V_{2s-1}$ then $uv_0v_1$ is not a 3-edge of $G_S$; also, if $u,u'$ lie in the same part $V_i$, $i \in [2s-2]$, then $uu'v_0$ is not a 3-edge of $G_S$.

We must thus have $v_1, v_2 \ldots v_{s}$ lie in $s$ distinct parts from $V_1,V_2, \ldots V_{2s-2}$ (or else we do not have a copy of $S(K_s^{(2)})$. But placing a vertex inside part $V_i$ for some $1\leq i \leq (s-1)$ forbids us from placing any vertex in part $V_{s-1+i}$, and vice-versa. Since $2\times s-1>\left\vert[2s-2]\right\vert=2s-2$, this contradicts the pigeon-hole principle. The 3-graph $G_S$ is thus $S(K_s^{(2)})$-free, as claimed.

Now the number of 3-edges contained in $G_S$ is
\begin{align*}
\vert E(G_S)\vert &= \left(\binom{2s-1}{3}-\vert E(S)\vert\right) \left(\frac{n}{2s-1}\right)^3\times \left(1+ (2s-1)\times \frac{1}{(2s-1)^3}\right.\\
&\left. \qquad  +(2s-1)^2\times \frac{1}{(2s-1)^6} + \cdots\right) +O(n^2)\\
&= \left(\frac{2s-4}{2s}+o(1)\right)\binom{n}{3},
\end{align*}
as required.
\end{proof}
\begin{corollary}\label{Turandensitysuspension}
Let $s\in \mathbb{N}$ be congruent to $1$ or $2$ modulo $3$. Then
\[\pi(S(K_s^{(2)}))\geq 1- \frac{2}{s}.\]
\end{corollary}
\begin{conjecture}\label{turansuspensionconj}
The lower bound given above in Corollary~\ref{Turandensitysuspension} is sharp.
\end{conjecture}
By Tur\'an's theorem and averaging, we have that 
\[\pi(S(K_s^{(2)}))\leq 1- \frac{1}{s-1}\]
for all $s \geq 2$. On the other hand Construction~\ref{iteratedsteiner} shows  that $\pi(S(K_s^{(2)}))\geq 1-\frac{1}{2s}+O\left(\frac{1}{s^2}\right)$. It seems likely that $\pi(S(K_s^{(2)}))=1-\frac{C}{s}+O\left(\frac{1}{s^2}\right)$ for some constant $C\in [1,2]$. As a first step towards a general result, it would be interesting to identify the correct value of $C$ (or indeed prove that such a constant exists!).
\begin{conjecture}[Weakening of Conjecture~\ref{turansuspensionconj}]
For every $\varepsilon>0$ there exists $s_0=s_0(\varepsilon)$ such that for all $s\geq s_0$ we have
\[\pi(S(K_s^{(2)}))\leq 1- \frac{2-\varepsilon}{s}.\]
\end{conjecture}
Similarly, we would like to know if 
\[\gamma(S(K_s^{(2)}))=1-\frac{C'}{s}+O\left(\frac{1}{s^2}\right)\]
for some constant $C'\in [1,3]$, and to know the value of $C'$ (if it does exist).

\section{Co-spanned complete 3-graphs}
In view of Tur\'an's theorem, there is another family of 3-graph analogues of complete 2-graphs we could want to study. Observe that a complete 2-graph on $s+1$ vertices $K_{s+1}^{(2)}$ may be viewed as a $K_s^{(2)}$ spanned by the neighbourhood of a single vertex.

By going over to 3-graphs and looking at the joint neighbourhood of a pair of vertices, we have the natural concept of a co-spanned 3-graph.
\begin{definition}
Given a $3$-graph $H$, let $F(H)$ denote the \emph{co-spanned} $H$, that is, the 3-graph with vertex set
$V(H)\sqcup\{x_{\star},y_{\star}\}$ and 3-edges
\[\{x_{\star}y_{\star}v: \ v \in V(H)\}\sqcup E(H).\]
\end{definition}
So for example the 3-graph $F_{3,2}=([5], \{123,124,125,345\})$ may be thought of as the co-spanned $K_3$.

As mentioned in the introduction Marchant, Pikhurko, Vaughan and the author~\cite{falgasravry+marchant+pikhurko+vaughan:arxiv} showed $\gamma(F_{3,2})=1/3$ and determined the extremal configurations, which are obtained by taking an almost balanced tripartition of the vertex set $\sqcup_{i=1}^3V_i$, putting in all triples of the form $V_iV_iV_{i+1}$ (winding round modulo 3) and adding a small number ($O(n^2)$) of 3-edges of the form $V_1V_2V_3$.

\begin{problem}\label{cospannedproblem}
Find a good lower bound construction for $\gamma(F(K_t))$.
\end{problem}
Clearly for any 3-graph $H$ we must have $\gamma(F(H))\geq \gamma(H)$. We also have the following upper-bound.
\begin{proposition}\label{spanprob}
\[\gamma(F(H))\leq \frac{1}{2-\gamma(H)}.\]
\end{proposition}
\begin{proof}
Assume that $\gamma(H)<1$, for otherwise we have nothing to prove. Let $\gamma$ be a real number with $\frac{1}{2-\gamma(H)}< \gamma <1$. Suppose $G$ is a 3-graph on $[n]$ with $\delta_2(G)\geq \gamma n$. Consider any pair $xy \in [n]^{(2)}$, and let $\Gamma(x,y)=\{z\in [n]: \ xyz \in E(G)\}$ be their joint neighbourhood.

Set $\alpha =\vert \Gamma(x,y) \vert/n$ By our codegree assumption, we have $\alpha \geq \gamma$. Also by the codegree assumption, every pair of vertices $z,z' \in \Gamma(x,y)$ must have at least $\gamma n - (1-\alpha) n$ neighbours inside $\Gamma(x,y)$.

Thus the restriction $G'=G[\Gamma(x,y)]$ of the 3-graph $G$ to the vertices in $\Gamma(x,y)$ has codegree
\begin{align*}
\frac{\delta_2(G')}{\vert V(G')\vert}&\geq 1- \frac{1-\gamma}{\alpha}\\
&\geq 1- \frac{1-\gamma}{\gamma} \qquad \textrm{since $\alpha \geq \gamma$.}
\end{align*}
Since $\gamma >\frac{1}{2-\gamma(H)}$, we have $1- \frac{1-\gamma}{\gamma}>\gamma(H)$. In particular if $n$ is large enough, $G'$ contains a copy of $H$, which, when taken with $x,y$ in $G$ gives rise to a copy of $F(H)$.

It follows that $\gamma(F(H))\leq \frac{1}{2-\gamma(H)}$ as claimed.
\end{proof}
Thus if the answer to 
Question~\ref{areboundstight} is positive (as we believe) and $\gamma(K_t)=1-1/(s-2)$ for all $s \geq 4$, then 
\[1-\frac{1}{s-2}\leq \gamma(F(K_s))\leq 1- \frac{1}{s-1}\]
for all $s\geq 3$.
\begin{question}
Is $\gamma(F(K_s))=1-\frac{1}{s}+O\left(\frac{1}{s^2}\right)$?
\end{question}

Let us finish by noting that the Tur\'an density analogue of Problem~\ref{cospannedproblem} is also open. The Tur\'an number and extremal configurations ('one-way' bipartite 3-graphs) for $F_{3,2}$ are known~\cite{FurediPikhurkoSimonovits05}, but we do not know of any constructions or conjectures for $\pi(F(K_s))$ besides the ones showing $\pi(K_s)\geq 1-\frac{4}{(s-1)^2}$.

\bibliographystyle{plain}
\bibliography{codegreebiblio}

\end{document}